\documentclass[a4paper, draft, 12pt]{amsproc}
\usepackage{amssymb}
\usepackage{amsmath,amssymb,ulem,color}

\usepackage[hyphens]{url} 
\usepackage[dvips]{graphicx} 
\usepackage{cmdtrack} 
\usepackage{mathrsfs}                       

\theoremstyle{plain}
 \newtheorem{theorem}{Theorem}[section]
 \newtheorem{prop}{Proposition}[section]
 \newtheorem{lem}{Lemma}[section]
 
\theoremstyle{Definition}
 \newtheorem{exm}{Example}[section]
 
 \newtheorem{dfn}{Definition}[section]
\theoremstyle{remark}

 \numberwithin{equation}{section}

\renewcommand{\leq}{\leqslant}
\renewcommand{\geq}{\geqslant}

\title[Transcendental Semigroup that has Simply Connected Fatou Components]{Transcendental Semigroup that has Simply Connected Fatou Components}

\subjclass[2010]{37F10, 30D05}

\keywords{Fatou set, pre-periodic component, periodic component, wandering component, transcendental semigroup, Carleman set}

\author[B. H. Subedi]{\bfseries  Bishnu Hari Subedi}

\address{ 
Central Department of Mathematics \\ 
Institute of Science and Technology   \\ 
Tribhuvan University   \\ 
Kirtipur, Kathmandu\\
Nepal}
\email{subedi.abs@gmail.com / subedi\_bh@cdmathtu.edu.np }

\author[A. Singh]{Ajaya Singh}
\address{Central Department of Mathematics, Institute of Science and Technology, Tribhuvan University, Kirtipur, Kathmandu, Nepal }
\email{singh.ajaya1@gmail.com / singh\_a@cdmathtu.edu.np} 
\thanks{This research work of first author is supported by PhD faculty fellowship from University Grants Commission, Nepal. } 

\begin{document}

{\begin{flushleft}\baselineskip9pt\scriptsize
\end{flushleft}}
\vspace{18mm} \setcounter{page}{1} \thispagestyle{empty}

\begin{abstract}
We prove that there exists a non-trivial transcendental semigroup $S$ such that the wandering (or pre-periodic or periodic) components of Fatou set $F(S)$ has at least a simply connected domain $D$.
\end{abstract}

\maketitle
\section{Introduction}
 We denote the \textit{complex plane} by $\mathbb{C}$, \textit{extended complex plane} by $\mathbb{C_{\infty}}$ and \textit{set of integers greater than zero} by $\mathbb{N}$. 
We assume the function $f:\mathbb{C}\rightarrow\mathbb{C}$ is \textit{transcendental entire function} (TEF) unless otherwise stated. 
For any $n\in\mathbb{N}, \;\; f^{n}$ always denotes the nth \textit{iterates} of $f$.
If $f^{n}(z) = z$ for some smallest $ n \in \mathbb{N} $, then we say that $z$ is periodic point of period n. In particular, if $f (z) = z$, then $z$ is a fixed point of $f$. If $| (f^{n}(z))^{\prime} | < 1$, where  $ \prime $ represents complex differentiation of $f^{n}$ with respect to $z$, then $z$ is called attracting periodic point. 
A family $\mathscr{F} = \{f:  f\;  \text{is meromorphic on some domain} \; X \; of \; \mathbb{C}_{\infty}\}$ forms normal family if every sequence $(f)_{i\in\mathbb{N}}$  of functions contains a subsequence which converges uniformly to a finite limit or converges to $ \infty $ on every compact subset $D$ of  $X$.  
The \textit{Fatou set}\index{Fatou ! set} of $f$ denoted by $F(f)$ is the set of points $z\in \mathbb{C}$ such that sequence $(f^n)_{n\in \mathbb{N}}$ forms a normal family in some neighborhood of $z$.  That is, $ z\in F(f) $ if $ z $ has a neighborhood $ U$ on which the family $\mathscr{F}$ is normal.  By definition, Fatou set is open and may or may not be empty.  Fatou set is non-empty for every entire function with attracting periodic points. If $ U \subset F(f) $ (Fatou component), then $ f(U) $ lies in some component $ V $ of $ F(f) $ and $ V- f(U) $ is a set which contains at most one point (see for instance \cite{her}). Let $ U \subset F(f) $ (a Fatou component) such that $ f^{n}(U) $ for some $ n \in \mathbb{N} $, is contained in some component of $ F(f) $, which is usually denoted by $ U_{n} $.  A Fatou component $ U $ is called pre-periodic if there exist integers $ n >m \geq 0 $ such that $ U_{n} =  U_{m} $. In particular,  if $ U_{n} =  U_{0} = U$ ( that is, $ f^{n}(U) \subset U $) for some smallest positive integer $ n \geq 1 $, then $ U $ is called \textit{periodic Fatou component}\index{periodic! Fatou component} of period n and $\{ U_{0}, U_{1}\ldots, U_{n-1} \}$ is called the \textit{periodic cycle}\index{periodic ! cycle} of $ U $.  In the case, if $ U_{1} = f(U) \subset U $, then $ U $ is called \textit{invariant domain}\index{invariant domain}. A component of Fatou set $ F(f) $ which is not pre-periodic is called \textit{wandering domain}\index{wandering domain}.  

For the complex plane $\mathbb{C}$, let us  denote the set of all holomorphic functions of $ \mathbb{C} $ by $\text{Hol}(\mathbb{C})$. If $ f\in \text{Hol}(\mathbb{C}) $, then $ f $ is a polynomial or transcendental entire function. The set $\text{Hol}(\mathbb{C})$ forms a semigroup  with semigroup operation being the functional composition. 
 
\begin{dfn}[\textbf{Transcendental semigroup}]
Let $ A = \{f_i: i=1,2,\ldots\} \subset \text{Hol}(\mathbb{C})$ be a set of transcendental entire functions $ f_{i}: \mathbb{C}\rightarrow \mathbb{C} $. A \textit{transcendental semigroup} $S$ is a semigroup generated by the set $ A $ with semigroup operation being the functional composition. We denote this semigroup by $S = \langle A \rangle =  \langle f_i: i= 1, 2, 3,\ldots\rangle$ or simply by $S = \langle f_{1}, f_{2}, f_{3}, \cdots, f_{n}, \cdots \rangle$. 
\end{dfn}
Here, each $f \in S$ is the transcendental entire function and $S$ is closed under functional composition. Thus $f \in S$ is constructed through the composition of finite number of functions $f_{i_k},\;  (k=1, 2, 3,\ldots, m) $. That is, $f =f_{i_1}\circ f_{i_2}\circ f_{i_3}\circ \cdots\circ f_{i_m}$. 

A semigroup generated by finitely many functions $f_i, (i = 1, 2, 3,\ldots, n) $  is called \textit{finitely generated transcendental semigroup}. We write $S= \langle f_1,f_2,\ldots,f_n\rangle$.
 If $S$ is generated by only one transcendental entire function $f$, then $S$ is \textit{cyclic or trivial transcendental  semigroup}. We write $S= \langle f\rangle$. In this case each $g \in S$ can be written as $g = f^n$, where $f^n$ is the nth iterates of $f$ with itself. Note that in our study of semigroup dynamics, we say $S= \langle f\rangle$ a trivial semigroup. 
 
Based on the Fatou-Julia theory of a complex analytic function, The Fatou set and Julia set  in the settings of semigroup are defined as follows.
\begin{dfn}[\textbf{Fatou set, Julia set}]\label{2ab} 
\textit{Fatou set} of the transcendental semigroup $S$ is defined by
  \[F (S) = \{z \in \mathbb{C}: S\;\ \textrm{is normal in a neighborhood of}\;\ z\}\] 
The connected component of Fatou set $ F(S) $ is called Fatou component.
The \textit{Julia set} of $S$ is defined by $J(S) = \mathbb{C} - F(S)$.
\end{dfn} 

From the definition \ref{2ab}, it is clear that $F(S)$ is the open set and therefore, it complement $J(S)$ is closed set. Indeed, these definitions generalize the definitions of Julia set and Fatou set of the iteration of single holomorphic map. 
If $S = \langle f\rangle$, then $F(S)$ and $J(S)$  are respectively the Fatou set and Julia set in classical iteration theory of complex dynamics. In this situation we simply write: $F(f)$ and $J(f)$. 

The fundamental contrast between classical complex dynamics and semigroup dynamics appears by different algebraic structure of corresponding semigroups. In fact, non-trivial semigroup (rational or transcendental) need not be, and most often will not be abelian. However, trivial semigroup is cyclic and therefore abelian. As we discussed before, classical complex dynamics is a dynamical study of trivial (cyclic) semigroup whereas semigroup dynamics is a dynamical study of non-trivial semigroup.

Note that for any semigroup $ S $, we have
\begin{enumerate}
\item $ F(S) \subset F(h) $ all $ h \in S$ and hence $ F(S) \subset \cap_{h\in S}F(h) $.  
\item $ J(h) \subset J(S) $.
\end{enumerate} 
Since, in classical complex dynamics, Fatou set $F(f)$ may be empty. So from the above first relation, we can say that Fatou set $F(S)$ of semigroup $S$ may also be empty. 
In this paper we are interested to find a non-trivial semigroup $S$ that has non-empty Fatou set $F(S)$. Basically, we prove that there is a non-trivial transcendental semigroup that has simply connected Fatou component. 

\begin{prop}\label{ne}
There is a non trivial transcendental semigroup $ S $ such that the Fatou set $ F(S) $ has at least a simply  connected component. 
\end{prop}
Note that if the semigroup $S $ is trivial, that is, semigroup $S =\langle f \rangle$ generated by a single transcendental entire function $f$, then Bergweiler\cite{ber10} proved that the Fatou set $F(S)$ has both a simply and a multiply connected wandering domains. However, in the case of non-trivial transcendental semigroup, the proof is not so easy. The reason behind is that the dynamics of individual transcendental entire functions differ largely from the dynamics of their composites.

\section{Some Essential Lemmas}
To workout a proof of the proposition \ref{ne}, first of all we need a notion of approximation theory of entire functions. In our case, we can use the notion of Carleman set \index{Carleman set}  from which we obtain  approximation of any holomorphic map by entire functions.
\begin{dfn}[\textbf{Carleman Set}]\label{cs}
Let $ F $ be a closed subset of $ \mathbb{C} $ and $ C(F) =\{f : F\to \mathbb{C}: f\; \text{is continuous on }\; S\; \text{and analytic in the interior of}\; F^{\circ}\; \text{of}\; F \}  $. Then $ F $ is called a Carleman set (for $ \mathbb{C} $) if for any $ g \in C(F) $ and any positive continuous function $ \epsilon $ on $ F $, there exists entire function $ h $ such that $ |g(z) -h(z)| < \epsilon $ for all $ z \in F $. 
\end{dfn}
The following important characterization of Carleman set has been proved by A. Nersesjan in 1971 but we have been taken this result from \cite{gai}.
\begin{theorem}[{\cite[Theorem 4, page 157]{gai}}]\label{cs1}
Let $ F $ be proper subset of $ \mathbb{C} $. Then $ F  $ is a Carleman set for $ \mathbb{C} $ if and only if  $ F $ satisfies:
\begin{enumerate}
\item $ \mathbb{C}_{\infty} - F$ is connected;
\item $ \mathbb{C}_{\infty} - F$ is locally connected at $ \infty $;
\item for every compact subset $ K $ of  $ \mathbb{C} $, there is a neighborhood $ V $ of $ \infty $ in $ \mathbb{C}_{\infty}$ such that no component of $ F^{\circ} $ intersects both $ K $ and $ V $.
\end{enumerate}
\end{theorem}
It is well known in classical complex analysis that the space $ \mathbb{C}_{\infty} - F$ is connected if and only if each component $ Z $ of open set $ \mathbb{C} - F$ is unbounded. This fact together with above
 theorem \ref{cs1} will be a nice tool whether a  set  is a Carleman set for $ \mathbb{C} $. The sets given in the following examples are Carlemen sets for $ \mathbb{C} $.

\begin{exm}[{\cite[Example-page 133]{gai}}]\label{cse}
The set $ E = \{z\in \mathbb{C} : |z| =1, \text{Re}z > 0 \} \cup \{z =x: x>1\} \cup \big(\bigcup_{n =3}^{\infty} \{z =r e^{i\theta}: r>1, \theta = \pi/n \}\big)$ is a Carleman set for $ \mathbb{C} $. 
\end{exm}

\begin{exm}[{\cite[Set S, page-131]{singh}}]\label{cs3}
The set $ E =G_{0} \cup \big(\bigcup_{k =1}^{\infty}(G_{k}\cup B_{K} \cup L_{k}\cup M_{k})\big)$, where $ G_{0} = \{z\in \mathbb{C}:|z-2| \leq 1\} $;
 \begin{eqnarray*}
 G_{k} &= &\{z\in \mathbb{C} : |z -(4k +2)| \leq 1 \} \cup \{z \in \mathbb{C} : \text{Re}z = 4k +2,\;  \text{Im}z \geq 1 \}  \\ && \cup \{z \in \mathbb{C} : \text{Re}z = 4k +2, \;  \text{Im}z \leq -1 \},\; \; (k =1,2,3,\ldots);
\end{eqnarray*} 
 \begin{eqnarray*}
 B_{k} & = &\{z\in \mathbb{C} : |z +(4k +2)| \leq 1 \} \cup \{z \in \mathbb{C} : \text{Re}z = -(4k +2),\;  \text{Im}z \geq 1 \} \cup \\ && \{z \in \mathbb{C} : \text{Re}z = -(4k +2), \;  \text{Im}z \leq -1 \},\; \; (k =1,2,3,\ldots);
\end{eqnarray*}
 $$
 L_{k} = \{z \in \mathbb{C} : \text{Re}z = 4k \},\; \; (k =1,2,3,\ldots); 
 $$
 and 
 $$  
 M_{k} = \{z \in \mathbb{C} : \text{Re}z = -4k \},\; \; (k =1,2,3,\ldots)
 $$ 
 is a Carleman set for $ \mathbb{C} $ by the theorem \ref{cs1} . 
 \end{exm}
From the help of the Carleman set of the example \ref{cs3}, A.P. Singh {\cite[Theorem 2]{singh}} proved the following result.
\begin{lem}\label{ne3}
There are transcendental entire functions $ f $ and $ g $ such that there exists a domain which lies in the wandering component of the $ F(f), F(g),  F(f\circ g) $ and $ F(g\circ f) $.
\end{lem}
In fact, A. P. Singh \cite{singh} also had proved other results regarding the dynamics of two individual functions and their composites (see for instance {\cite[Theorem 1, Theorem 3 and Theorem 4] {singh}}) which are also stricly based on the Carleman set of example \ref{cs3}. Dinesh Kumar, Gopal Datt and Sanjay Kumar  extended  these result of A.P. Singh in {\cite[Theorem 2.1 to Theorem 2.15]{kum8}}. For our purpose, we cite the following two results from \cite{kum8}. 

\begin{lem}[\textbf{Theorem 2.2}] \label{ne2}
There are TEFs $ f $ and $ g $ such that there exists infinite number of domains which lie in the wandering component of the $ F(f),  F(g)$, $ F(f\circ g) $ and $ F(g\circ f) $.
\end{lem}
\begin{lem} [\textbf{Theorem 2.13}]\label{ne1}
There are TEFs $ f $ and $ g $ such that there exist  infinite number of domains which lie in the pre-periodic component of the $ F(f), F(g)$, $F(f\circ g) $ and $ F(g\circ f) $.
\end{lem}

We extended above  lemmas \ref{ne3}. \ref{ne1} and \ref{ne2} in {\cite[Theorem 1.1]{sub1}}{\cite[Theorem 1.1]{sub2}} and {\cite[Theorem 1.1]{sub3}} to the following results:
\begin{lem}\label{ne4}
There are transcendental entire functions $ f $, $ g $ and $ h $ such that there exist  infinite number of domains which lie in the wandering component of the $ F(f), F(g),  F(h),  F(f\circ g),  F(g\circ f),  F(f\circ h),  F(g\circ h),  F(h\circ f), F(h\circ g) , F(f\circ g \circ h), F(f\circ h \circ g), F(g\circ f \circ h), F(g\circ h \circ f), F(h\circ f \circ g)$ and $F(h\circ g \circ f)$. 
\end{lem}
 
\begin{lem}\label{ne6}
There are transcendental entire functions $ f $, $ g $ and $ h $ such that there exist  infinite number of domains which lie in the pre-periodic component of the $ F(f), F(g), F(h), F(f\circ g), F(g\circ f),  F(f\circ h),  F(g\circ h),  F(h\circ f),  F(h\circ g) ,  F(f\circ g \circ h), F(f\circ h \circ g), F(g\circ f \circ h), F(g\circ h \circ f), F(h\circ f \circ g)$ and $F(h\circ g \circ f)$. 
\end{lem}

\begin{lem}\label{ne5}
There are transcendental entire functions $ f $, $ g $ and $ h $ such that there exist  infinite number of domains which lie in the periodic component of the $ F(f),\; F(g), \; F(h), \; F(f\circ g), \; F(g\circ f), \; F(f\circ h), \; F(g\circ h), \; F(h\circ f), \; F(h\circ g), \; F(f\circ g \circ h),\;F(f\circ h \circ g),\;F(g\circ f \circ h),\;F(g\circ h \circ f),\;F(h\circ f \circ g)$ and $F(h\circ g \circ f)$. 
\end{lem}

\section{Proof of the Proposition \ref{ne}}
From all of above lemmas (Lemmas \ref{ne3}, \ref{ne1}, \ref{ne2}, \ref{ne4}, \ref{ne6} and \ref{ne5}),  we can say that whatever domains that lie in the wandering, pre-periodic or periodic components of $ F(f), F(g)$, $F(h), F(f\circ g),F(g\circ f), F(f\circ h),  F(g\circ h),  F(h\circ f),  F(h\circ g) ,  F(f\circ g \circ h), F(f\circ h \circ g), F(g\circ f \circ h), F(g\circ h \circ f), F(h\circ f \circ g)$ and $F(h\circ g \circ f)$, they also lie respectively in the wandering, pre-periodic or periodic components of their successive composites. In this context, we can also prove the following two results:
\begin{lem}\label{ne7}
If $ D $ is a set which lies in the wandering (or pre-periodic  or periodic) component of $ F(f), F(g),  F(f\circ g)$ and $ F(g\circ f)$, then it  also lies in wandering (or pre-periodic or periodic) component of $ F(f^{n_{k}} \circ g^{n_{k-1}}\circ \ldots \circ g^{n_{1}})$ and  $F(g^{n_{k}} \circ f^{n_{k-1}}\circ \ldots \circ f^{n_{1}})$, where $ n_{k}, \ldots n_{1}\in \mathbb{N} $.
\end{lem}
 \begin{proof}
 By the lemmas \ref{ne3}, \ref{ne2} and \ref{ne1}, such a set $ D $ exists. Since $ F(f) = F(f^{n}) $ and $ F(g) = F(g^{n}) $ for any $ n \in \mathbb{N} $. So $ D $ lies in the wandering (or pre-periodic or periodic) component of  $F(f^{n}) $ and $  F(g^{n}) $ for all $ n \in \mathbb{N} $.   As $ D $ lies in the  wandering (or pre-periodic or periodic) component of $F(f\circ g)$, it also lies in the wandering (or pre-periodic or periodic) component of $F(f^{n}\circ g^{n})$ for all $ n \in \mathbb{N} $. By the same argument we are using here, $ D $ also lies in the wandering (or pre-periodic or periodic) component of $F(f\circ g)^{n}$ for all $ n \in \mathbb{N} $. Since $F(f\circ g)^{n} = F(f \circ g \circ \ldots  \circ f \circ g)$ (n -times $ f \circ g $), $ D $ lies in the wandering (or pre-periodic or periodic) component of $F(f^{n}\circ g^{n} \circ \ldots  \circ f^{n} \circ g^{n})$(n-times $ f^{n} \circ g^{n} $) for all $ n \in \mathbb{N} $. Since $ n \in \mathbb{N} $ is arbitrary, so we conclude that $ D $ lies in the wandering (or pre-periodic or periodic) component of $ F(f^{n_{k}} \circ g^{n_{k-1}}\circ \ldots \circ g^{n_{1}})$ for all $ n_{k}, \ldots n_{1}\in \mathbb{N}$. Similarly, we can show that  $ D $ lies in the wandering (or pre-periodic or periodic) component of $F(g^{n_{k}} \circ f^{n_{k-1}}\circ \ldots \circ f^{n_{1}})$ for all $ n_{k}, \ldots n_{1}\in \mathbb{N} $.
 \end{proof}
 
 \begin{lem}\label{ne8}
If $ D $ is a set which lies in the wandering (or pre-periodic or periodic) component of $ F(f), F(g),  F(h)$, $F(f\circ g)$, $F(g\circ f), \; F(f\circ h), \; F(g\circ h), \; F(h\circ f), \; F(h\circ g) , \; F(f\circ g \circ h),\;F(f\circ h \circ g),\;F(g\circ f \circ h),\;F(g\circ h \circ f),\;F(h\circ f \circ g)$ and $F(h\circ g \circ f)$,  then it also lies in the  wandering (pre-periodic or periodic) component of $ F(f^{n_{k}} \circ g^{n_{k-1}}\circ h^{n_{k-2}} \ldots \circ f^{n_{1}})$,  $F(g^{n_{k}} \circ f^{n_{k-1}} \circ h^{n_{k-2}}\circ \ldots \circ g^{n_{1}})$ and $F(h^{n_{k}} \circ f^{n_{k-1}} \circ g^{n_{k-2}}\circ \ldots \circ h^{n_{1}})$ etc.
\end{lem}
 \begin{proof}
 By lemmas  \ref{ne4}, \ref{ne6}, \ref{ne5}, such a set $ D $ exists. By the similar argument of above lemma \ref{ne7}, the proof of this lemma follows.
 \end{proof}
 
 We prove the proposition \ref{ne} for a semigroup generated by two or three transcendental entire functions as defined in above lemmas \ref{ne3}, \ref{ne2}, \ref{ne1},  \ref{ne4}, \ref{ne6} and \ref{ne5}.

\begin{proof}[Proof of the Proposition \ref{ne}]
Let $ S $ be a transcendental semigroup generated by two or three transcendental entire functions. If  $ S $ is generated by two transcendental entire functions $ f $ and $ g $ as defined in the lemmas \ref{ne3},  \ref{ne2} and \ref{ne1}, then by the lemma \ref{ne7}, there is at least a domain which lies in the wandering (or pre-periodic or periodic) component  of the  $ F(f^{n_{k}} \circ g^{n_{k-1}}\circ \ldots \circ g^{n_{1}})$ and  $F(g^{n_{k}} \circ f^{n_{k-1}}\circ \ldots \circ f^{n_{1}})$ for all $ n_{k}, \ldots n_{1}\in \mathbb{N} $. By the definition of transcendental semigroup, any $ h \in S = \langle f, g \rangle $ can be written in either of the form $ h = f^{n_{k}} \circ g^{n_{k -1}} \circ \ldots \circ g^{n_{1}} $ or $ h = g^{n_{k}} \circ f^{n_{k -1}} \circ \ldots \circ f^{n_{1}} $ for all $ n_{k}, \ldots n_{1}\in \mathbb{N} $. Therefore, there is a domain $ D $ which lies in the wandering (or pre-periodic or periodic) component  of the Fatou set $ F(h) $ for every function $ h $ of transcendental semigroup $ S $. This shows that this domain lies in the wandering (or pre-periodic or periodic) component  of the Fatou set $ F(S) $. Since for transcendental entire function,  pre-peridic (or periodic) domains are simply connected and so a domain within simply connected domains is also simply connected.  In the construction of functions  in the lemmas \ref{ne3},  \ref{ne2} and \ref{ne1}, the domain which lies in the wandering domains is simply connected. If  $ S $ is generated by three transcendental entire functions $ f $, $ g $ and $ h $ as defined in the lemmas \ref{ne4},  \ref{ne6} and \ref{ne5}, then by lemma \ref{ne8} and similar argument as above,  Fatou set $ F(S) $ contains  simply connected domain. 
\end{proof}

We restricted our proof of the proposition \ref{ne} to the transcendental semigroup generated by  two or three transcendental entire functions . Rigorously, it is not known that the  essence of this proposition holds if a semigroup is generated by more than three transcendental entire functions. We can only say intuitively, the essence of this proposition  may hold if  semigroup $ S $ is generated by n- transcendental entire functions. There is another strong aspect of this proposition which is- if a transcendental semigroup  $ S $ generated by such type of two or three transcendental entire functions, then the Fatou set $ F(S) $ is non-empty.

\end{document}